\documentclass[12pt]{article} 
  
\usepackage{amsmath}
\usepackage{amsfonts}
\usepackage{amsthm}     
\usepackage{amssymb}  
\usepackage{latexsym}
\usepackage{euscript}   
    
\newcommand{\msf}[1]{\mathsf {#1}}

\newcommand{\mcal}[1]{{\mathcal {#1}}} 

\newtheorem{theorem}{Theorem}   [section]
\newtheorem{lemma}[theorem]{Lemma}
\newtheorem{corollary}[theorem]{Corollary}
\newtheorem{proposition}[theorem]{Proposition}

\theoremstyle{definition}
\theoremstyle{definition}\newtheorem{definition}[theorem]{Definition}
\theoremstyle{definition}
\theoremstyle{definition}

\theoremstyle{definition}
\theoremstyle{definition}

\theoremstyle{plain}\newtheorem*{theorem*}{Theorem}
\theoremstyle{plain}\newtheorem*{corollary*}{Corollary}
\theoremstyle{remark}\newtheorem*{remark*}{Remark}
\theoremstyle{remark}\newtheorem*{remarks*}{Remarks}
\theoremstyle{definition}\newtheorem*{conjecture*}{Conjecture}
\theoremstyle{definition}\newtheorem*{definition*}{Definition}
\theoremstyle{definition}\newtheorem*{example*}{Example}

\theoremstyle{definition}\newtheorem*{question*}{Question}
\theoremstyle{definition}\newtheorem*{questions*}{Questions}

\theoremstyle{definition}\newtheorem*{hypothesis*}{Hypothesis}

\def\qed{\ifhmode\unskip\nobreak\fi\ifmmode\ifinner\else\hskip5pt\fi\fi
 \hfill\hbox{\hskip5pt\vrule width4pt height6pt depth1.5pt\hskip1pt}}
 
\newcommand{\ov}[1]{\ensuremath{\overline{#1}}}

\newcommand{\Int}{\ensuremath{\operatorname{\mathsf{Int}}}}

\newcommand{\RR}{\ensuremath{{\mathbb R}}}     
\newcommand{\R}[1]{\ensuremath{{\mathbb R}^{#1}}} 

\newcommand{\FF}{\ensuremath{{\mathbb F}}}     

\newcommand{\F}[1]{\ensuremath{{\mathbb F}^{#1}}} 
\newcommand{\C}[1]{\ensuremath{{\mathbb C}^{#1}}} 

\newcommand{\CC}{\ensuremath{{\mathbb C}}}     

\newcommand{\ZZ}{\ensuremath{{\mathbb Z}}}	   
\renewcommand{\S}[1]{\ensuremath{{\bf S}^{#1}}} 

\newcommand{\om}[1]{\ensuremath{\omega(#1)}}
\newcommand{\emp}{\ensuremath{\varnothing}}

\newcommand {\pd}[1] {\ensuremath{\frac{\partial}{\partial #1}}}
\newcommand{\ode}[2]{\ensuremath{\frac{d#1}{d#2}}}    	

\def\d#1dt{\frac{d#1}{dt}}    

\newcommand{\eps}{\ensuremath{\epsilon}}
\newcommand{\lam}{\ensuremath{\lambda}}
\newcommand{\Lam}{\ensuremath{\Lambda}}

\newcommand{\Sig}{\ensuremath{\Sigma}}

\newcommand{\Fix}[1]{\ensuremath{\operatorname{\mathsf {Fix}}(#1)}}
\newcommand{\co}{\colon\thinspace} 

\def\empty{\varnothing}

\newcommand{\fr}[1]{\ensuremath{\operatorname{\mathsf{Fr}}(#1)}}

\renewcommand{\ss}{\ensuremath{\mathfrak {s}}}

\newcommand{\sm}[1]{\ensuremath{\setminus #1}}

\newcommand{\p}{\ensuremath{\partial}}

\def\empty{\varnothing}
\def\emp{\varnothing}
\reversemarginpar   

	\def\mylabel#1{\label{#1}}
\newcommand{\V}{\ensuremath{{\mcal V}}}
	
\newcommand{\Z}[1]{\ensuremath{{\msf  Z ( #1)}}}
\renewcommand{\ss}  {\ensuremath{{\mathfrak {s}}}}
\renewcommand{\dim} {\ensuremath{{\mathsf {dim}}}}
\renewcommand{\om}{\omega}

\def\emp{\varnothing}
\begin{document}


\title{\bf Zero sets of Lie algebras of analytic vector fields on real
  and complex 2-dimensional manifolds} \author{{\bf Morris
    W. Hirsch}\thanks{I thank Paul Chernoff, David Eisenbud, Robert
    Gompf, Henry King, Robion Kirby, Frank Raymond, Helena Reis and
    Joel Robbin for helpful discussions.}  \\ Mathematics
  Department\\ University of Wisconsin at Madison\\ University of
  California at Berkeley}
\maketitle

\begin{abstract} 
Let $M$ denote a real or complex analytic manifold with empty
boundary, having dimension $2$ over the ground field $\FF =\RR$ or
$\CC$.  Let $Y, X$ denote analytic vector fields on $M$.  Say that $Y$
{\em tracks} $X$ provided 
$[Y,X] = fX$ with $f\co M\to \FF$
continuous.  Let $K$ be a compact component of the zero set $\Z X$
whose Poincar\'e-Hopf index is nonzero.

\smallskip
\noindent
{\bf Theorem.} If $Y$ tracks $X$ then $\Z Y \cap K\ne\varnothing$. 

\smallskip
\noindent
{\bf Theorem.} Let $\mcal G$ be a Lie algebra of analytic vector
fields that track $X$, with $\mcal G$ finite-dimensional and
supersolvable when $M$ is real.  Then $\bigcap_{Y\in \mcal G}\Z Y$
meets $K$.

\end{abstract}

\tableofcontents
\section{Introduction}   \mylabel{sec:intro}
 Let $M$ denote a manifold with empty boundary $\p M$.  A fundamental
 issue in Dynamical Systems is deciding whether a vector field has a
 zero.
When is $M$ compact with  Euler characteristic $\chi (M)\ne 0$,
positive answer is given by the celebrated {\sc   Poincar\'e-Hopf} Theorem. 

Determining whether two or more vector fields have a common zero is
much more challenging.  For two commuting analytic vector fields on a
possibly noncompact manifold of dimension $\le 4$, Theorem
\ref{th:bonatti}, due to {\sc C. Bonatti}, not only gives conditions
sufficient for common zeros, it locates specific compact sets.  Our
results establish common zeros for other sets of vector fields.  The
basic results are Theorems  \ref{th:MAIN} and   \ref{th:liealg}.
Theorem \ref{th:liegroup} is an application to fixed points of
transformnation groups. 

Denote the zero set of a vector field $X$ on $M$  by $\Z X$. 
 A {\em block} of zeros for $X$, or an {\em $X$-block}, is a compact
 set $K\subset \Z X$ having an open neighborhood $U\subset M$ such
 that $\ov U$ is compact and $ \Z X \cap \ov U=K$.  Such a set $U$ is
 {\em isolating} for $X$ and for $(X, K)$.

 The {\em index} of the $X$-block $K$ is the integer $\msf i_K
 (X):=\msf i (X,U)$ defined as the Poincar\'e-Hopf index of any
 sufficiently close approximation to $X$ having only finitely many
 zeros in $U$ (Definitions \ref{th:defph}, \ref{th:defindex}).\footnote{
Equivalently: $\msf i (X, U)$ is the
intersection number of $X|U$ with the zero section of the tangent
bundle ({\sc Bonatti} \cite{Bonatti92}).}
This number is independent of $U$, and is stable under perturbations
of $X$ (Theorems \ref{th:stability}, \ref{th:approx}).
When $X$ is $C^1$ and generates the local flow $\phi$, for
sufficiently small $t >0$ the index $\msf i (X, U)$ equals the
fixed-point index $I (\phi_t|U)$ defined by {\sc Dold} \cite {Dold72}.

An $X$-block $K$ is {\em essential} if $\msf i_K (X)\ne 0$.   This
implies  $\Z X\cap K\ne\varnothing$ because every isolating neighborhood
of $K$ meets $\Z X$.

 \begin{theorem} [{\sc Poincar\'e-Hopf}]        \mylabel{th:PH}
If $M$ is compact $\msf i (X, M) =\chi (M)$ for all continuous vector
fields $X$ on $M$.
 \end{theorem}
 For calculations of the index in more general settings see {\sc Morse
 \cite{Morse29}, Pugh \cite{Pugh68}, Gottlieb \cite{Gottlieb86}, Jubin
 \cite{Jubin09}}.

This paper was inspired by a remarkable result of C. Bonatti, which
does not require compactness of $M$:\footnote{
 ``The demonstration of this result involves a beautiful and quite
  difficult local study of the set of zeros of $X$, as an analytic
  $Y$-invariant set.''\ ---{\sc P.\ Molino} \cite{Molino93}}

\begin{theorem} [{\sc  Bonatti} \cite{Bonatti92}]  \mylabel{th:bonatti}
Assume $M$ is a real manifold of dimension $\le 4$ and $X, Y$ are
analytic vector fields on $M$ such that $[X, Y]= 0$.
Then $\Z Y$ meets every essential $X$-block.\footnote{
In \cite{Bonatti92} this is stated for $\dim (M) = 3$ or $4$.  If
$\dim (M) =2$ the same conclusion is obtained by applying the
3-dimensional case to the vector fields $X \times t\pd t, \ Y\times
t\pd {t}$\, on $M\times\RR$.}
\end{theorem}

\subsection{Statement of results}   \mylabel{sec:results}
Theorems \ref{th:MAIN} and  \ref{th:liealg} extend Bonatti's Theorem
to certain pairs of noncommuting vector fields, and Lie algebras of
vector fields.  

$M$ always denotes a real or complex manifold without boundary, over
the corresponding ground field $\FF=\RR$ (the real numbers) or $\CC$
(the complex numbers).
In the main results $M$ has dimension $\dim_\FF M=2$ over $\FF$.

 $\V (M)$ is the vector space over $\FF$ of continuous vector
fields on $M$, with the compact-open topology.  The subspace of $C^r$
vector fields is $\ V^r (M)$, where $r$ is a positive integer,
$\infty$ (meaning $C^k$ for all finite $k$), or $\om$ (analytic over
$\FF$).  When $M$ is complex, $\mcal V^r (M)$ and $\mcal V^\om (M)$
are identical as real vector spaces, and  $\mcal V^\om (M)$ is a
complex Lie algebra.  A linear subspace of $\mcal V^r
(M)$ is called a Lie algebra if it is closed under Lie brackets.  The
set of common zeros of a set $\ss\subset \mcal V (M)$ is $\Z\ss
:=\bigcap_{X\in\ss}\Z X$.

$X$ and $Y$ denote vector fields on $M$. 

$Y$ {\em tracks $X$} provided $Y, X \in \V^1 (M)$ and $[Y, X]=fX$
with $f\co M\to \FF$ continuous.  When $\FF=\RR$ this implies the local flow
generated by $Y$ locally permutes orbits of $\Phi^X$ (Proposition
\ref{th:orbits}).  We say that a set of vector fields tracks $X$ when
each of its elements tracks $X$.

\begin{example*}[\bf A]    
Suppose $\mcal G$ is a Lie algebra of $C^1$ vector fields on $M$.  If
$X\in \mcal G$ spans an ideal then $ \mcal G$ tracks $X$, and the
converse holds provided $\mcal G$ is finite dimensional.
\end{example*}

Here are  the main results. The manifold  $M$ is real or complex with
$\dim_\FF M=2$. 
\begin{theorem}         \mylabel{th:MAIN}
Assume $X, Y\in \V^\om (M)$,  $Y$ tracks $X$, and $K\subset \Z X$ is an
essential $X$-block.  Then $\Z Y\cap  K\ne\varnothing$. 
\end{theorem}

A Lie algebra $\mcal G$ is {\em supersolvable} if it
is real and faithfully represented by  upper triangular
matrices.  If $\mcal G$ is the Lie algebra of Lie group $G$ we call $G$ 
supersolvable.

\begin{theorem}               \mylabel{th:liealg}
Assume $X\in\V^\om (M)$, $K$ is an essential $X$-bloc, and 
 $\mcal G\subset \V^\om (M)$ is a Lie algebra  that tracks $X$. 
Let
one  of the following conditions hold:
\begin{description}

\item[(a)] $M$ is complex,

\item[(b)] $M$ is real and $\mcal G$ is supersolvable. 
\end{description}
Then  $\Z{\mcal G}\cap K\ne\varnothing$.
\end{theorem}

\begin{example*}[\bf B]        
Let $M, \mcal G, X$ be as in Theorem \ref{th:liealg} and assume the
local flow $\Phi^X$ has a compact global attractor.  It can be shown
that $M$ is an isolating neighborhood for the $X$-block $\Z X$, $M$
has finitely generated homology, and $\msf i (X, M)=\chi (M)$.
 Theorem \ref{th:liealg} thus implies:
 
\begin{itemize}   
\item {\em If $\chi (M)\ne 0$ then $\Z{\mcal G}\ne\varnothing$.}
\end{itemize}
For instance:
\begin{itemize}   
\item
{\em Let $\mcal G$ be a Lie algebra of holomorphic vector fields
  on $\C 2$.  If $X\in \mcal G$ spans an ideal and $\Phi^X$ has a global attractor,
  then $\Z{\mcal G}\ne\varnothing$.
}
\end{itemize}
\end{example*}

Now let $G$ denote  a connected
Lie group over the same ground field as 
$M$.  
An  {\em analytic action} of $G$ on $M$ is a  homomorphism 
$\alpha\co g\mapsto g^\alpha$ from $G$ to the group of analytic 
diffeomorphisms of $M$; this action is also denoted by $(\alpha, G, M)$.
The action
is {\em effective} if its kernel 
is trivial.  Its {\em fixed point set} is 
\[\Fix \alpha:=\{p\in M\co g^\alpha (p)=p, \quad (g\in G)\}.\]

\begin{theorem}         \mylabel{th:liegroup} 
Assume:
\begin{itemize}
\item $M$ is a compact complex 2-manifold and 
  $\chi (M)\ne 0$,

\item $(\alpha, G, M)$ is  an  effective analytic action,

\item $G$ contains a  1-dimensional normal subgroup.
\end{itemize}
Then $\Fix \alpha \ne\varnothing$.

\end{theorem}
\begin{proof}
This follows from Theorem \ref{th:liealg}(a) (see Section
\ref{sec:proofs}). 
\end{proof}

The analogous result for  analytic actions of supersolvable Lie groups on
real surfaces  is due to {\sc Hirsch
  \& Weinstein} \cite{HW2000}.   But  {\sc Lima} \cite {Lima64} and 
{\sc Plante} \cite{Plante86} have shown that every compact surface supports 
    a  continuous fixed-point free 
  action by the 2-dimensional group whose Lie algebra has the
  structure $[Y, X]= X$.  Whether $X$ and $Y$ can be smooth is unknown.

\subsection{Terminology}
The closure of subset $\Lam$ of a topological space $S$ is denoted by
$\ov \Lam$, the frontier by $\fr \Lam:=\ov \Lam \cap \ov{S \sm\Lam}$,
and the interior by $\Int (\Lam)$.

Maps are continuous unless otherwise characterized.  A map is {\em
  null homotopic} if it is homotopic to a constant map.

If is vector $\xi$ in $\R n$, or a tangent vector to a Riemannian
manifold, its norm is $\|\xi\|$.  The unit sphere in $\R n$ is $\S n$.

Let $\dim_\FF (M)=n$.  The tangent bundle $\tau (M)$ is an $\F
n$-bundle, meaning a fibre bundle over $M$ with total space $T(M)$,
projection $\pi_M\co T(M)\to M$, standard fibre $\F n$, and structure
group $GL (n,\FF)$ (see Steenrod \cite{Steenrod51}).  The fibre over
$p\in M$ is tangent space to $M$ at $p$ is $T_p(M):=\pi_M^{-1}(p)$.
When $M$ is an open set in $\F n$ we identify $T_p(M)$ with $\F n$,
$\tau (M)$ with the trivial vector bundle $M\times \F n \to M$, and
$X\in\V (M)$ with the map \ $M\to \F n$, \,$p\mapsto X_p$.

Assume $X\in \V^r (M)$ and $\p M=\emp$. The local flow on $M$ whose
trajectory through $p$ is  the $X$-trajectory of $p$  is denoted by
$\Phi^X:=\big\{\Phi^X_t\big\}_{t\in\RR}$, referred to informaly as the {\em
  $X$-flow}.   The maps
$\Phi^X_t$ are $C^r$ diffeomorphisms between open subsets of $M$.

An {\em $X$-curve}  is the image of an integral
  curve $t\mapsto y(t)$ of $X$.  This is either a singleton and hence
  a zero of $X$, or an
  interval.  The maximal $X$-curve
  through $p$ is the {\em orbit} of $p$ under $X$. 
A set $S\subset M$ is {\em $X$-invariant}  if it contains the
orbits under $X$ of its points.  When this holds for all $X$ in a set
$\mcal H \subset \V^1 (M)$ then $S$ is {\em $\mcal H$-invariant}.

If $X, Y$ are vector fields on $M$, the alternating tensor field
$X\wedge Y\in \Lam^2 (M)$ may be denoted by $X\wedge_\FF Y$ in order
to emphasize the ground field.   $X\wedge_\FF Y=0$ means $X_p$ and
$Y_p$ are linearly dependent over $\FF$ at all $p\in M$.

\section{Consequences of tracking} \mylabel{sec:tracking}
Throughout this section we assume:
\begin{itemize}

\item {\em $M$ is a real or complex manifold,  \,$\dim_\FF (M)=n\ge 1$,
  \  $\p M=\emp$,} 

\item {\em $X, Y \in \V^1 (M)$,}

\end{itemize}
where $\FF$ denotes the ground field $\RR$ or
$\CC$. 

 \begin{definition*}              \mylabel{th:deftracks}
 $Y$ {\em tracks $X$} provided $Y, X \in \V^1 (M)$ and $[Y, X]=fX$
  with $f\co M\to \FF$ continuous.
 \end{definition*}

\begin{proposition}             \mylabel{th:backprop}
Suppose $Z \in \V^1 (M)$.  If $Y$ and $Z$ track $X$ and $[Y, Z]$ is
$C^1$, then $[Y, Z]$ tracks $X$.
\end{proposition}
\begin{proof} Follows from the Jacobi identity.
\end{proof}

\begin{definition}              \mylabel{th:defdepend}
The {\em dependency set} of $X$ and $Y$ (over the ground field) is 
\[
  \msf{Dep}_\FF (X,Y):=\big\{p\in M\co X_p\wedge_\FF Y_p=0\big\}. 
\]
\end{definition}
\begin{proposition}             \mylabel{th:ideal}
If $Y$ tracks $X$ then $\Z X$ and $\mcal D (X, Y)$ are $X$- and $Y$-invariant.
\end{proposition}
\begin{proof}
As the statement is local, we assume $M$ is an open set in $\F n$.

{\em Invariance of $\Z X$: } Evidently $\Z X$ is $X$-invariant and $\Z
X\cap \Z Y$ is  $Y$-invariant.
To show that
$\Z X\,\verb=\=\,\Z Y$ is $Y$-invariant, fix $p\in \Z X\,\verb=\=\,\Z Y $.
Let $(y_1,\dots,y_n)$ be flowbox coordinates in a neighborhood $V_p$
of $p$, representing $Y|V_p$ as $\pd{y_1}$ in a convex open subset of
$\R n$, and the $Y$-trajectory of $p$ as
\[ t\mapsto y(t):=p + te_1\]
where $e_1,\dots, e_n \in\F n$ are the standard basis vectors.

  Let $J_p\subset \RR$ be an open interval around $0$ such that
\[y(t)\in V_p, \qquad  (t\in J_p).\]
Then 
\begin{equation}                \label{eq:tpys}
 \ode{~}t\big(T\Phi^Y_t (X_p)\big) = [Y, X]_{y(t)}, \qquad (t\in J_p).
\end{equation}
Since $Y$ tracks $X$, there is a continuous $\FF$-valued function
$t\mapsto g(t)$ such that in the flowbox coordinates for $Y$, the
vector-valued function $t\mapsto X_{y(t)}$ satisfies the linear
initial value problem
\begin{equation}                \label{eq:tphi1}
 \ode{~}t X_{y(t)} = g(t) X_{y(t)}, \qquad X_p=0.
\end{equation}
Therefore $X_{y(t)}$ vanishes identically in $t$. 

{\em Invariance of $\msf D(X, Y)$: }
We need to prove:  for all $t\in J_p$,
\begin{equation}                \label{eq:tphi2}
X_p\wedge_\FF Y_p = 0 \implies  T\Phi^Y_t (X_p)\wedge_\FF X_{\Phi^Y_t(p)} =0.
 \end{equation}
Assume $Y_p\ne 0$ and fix flowbox
coordinates for $Y$ at $p$. 
It suffices to verify (\ref{eq:tphi2}) for all $t\in J_p$.  Equations
(\ref{eq:tpys}) and (\ref{eq:tphi1}) imply
\[ 
\begin{split}
\ode{~}t \big(T\Phi^Y_t (X_p)\wedge_\FF X_{y(t)}\big)
      &= \big([Y, X]\wedge_\FF X\big)_{y(t)} +
         T\Phi^Y_t (X_p)\wedge_\FF g(t) X_{y(t)}\\  
     &= 0 \ \text{ identically in $t$.}
\end{split}
\]
As (\ref{eq:tphi2}) holds for $t=0$, the proof is complete.
\end{proof}

The proof of the following result is similar and left to the reader:
\begin{proposition}               \mylabel{th:orbits}
If $M$ is real and $Y$ tracks $X$, each map $\Phi^Y_t$ sends orbits of
$X|\mcal D \Phi^Y_t$ to orbits of $X|\mcal R \Phi^Y_t$.\qed
\end{proposition}
When $M$ is complex
 there is a simliar result for the holomorphic local
actions of $\CC$ on $M$ generated by  $X$ and $Y$.
\section{The index function}   \mylabel{sec:indices}
In this section $M$ is a real surface of dimension $n\ge 1$ with
empty boundary.  
Assume $X\in \V (M)$, $K$ is an $X$-block, and the
precompact open set $U\subset M$ is isolating for $(X, K)$.
\begin{definition}              \mylabel{th:defdeform}
A {\em  deformation} from
$X$ to $X'$ is path in $\V (M)$ of the form 
\[
  t\mapsto X^t, \quad X^0:=X, \quad X^1= Y, \qquad (0\le t \le 1).
\]
The deformation is {\em nonsingular} in a set $S\subset M$ provided
$\Z{X^t}\cap S=\empty$.  
\end{definition}

\begin{proposition}             \mylabel{th:convex}
$X$ has arbitrarily small convex open neighborhoods $\mcal B\subset\V
  (M)$ such that for all $Y, Z\in \mcal B$:
\begin{description}

\item[(i)] $U$ is isolating for  $Y$,

\item[(ii)]  the deformation $Y^t:= (1-t)Y + tZ,\ (0\le t\le 1)$ is
  nonsingular in $\fr U$.  

\end{description}
These conditions imply:
\begin{description}

\item[(iii)] the set of $Y\in \mcal B$ such that 
$\Z Y\cap U$ is finite contains a dense open subset \\ of $\mcal B$.
\end{description}
\end{proposition}
\begin{proof}
(i) and (ii) follow from the definition of the compact-open
  topology on $\V (M)$.  Standard approximation theory gives  (iii). 
\end{proof}

\begin{definition}              \mylabel{th:defph}
When $K$ is finite, the {\em Poincar\'e-Hopf index} of $X$ at $K$, and
in $U$, is the integer $i^{PH}_K(X)=\msf i^{PH} (X, U)$ defined as
follows. 
 For each $p\in K$ choose an open set $W\subset U$ meeting
$K$ only at $p$, such that $W$ is
the domain of a  
$C^1$ chart
\[\phi\co W\approx W' \subset \R n, \quad \phi(p)=p'.
\]
 The transform of
 $X$ by $\phi$ is 
\[X':=T\phi\circ X \circ \phi^{-1} \in \V (W').
\]
There is a unique map of pairs
\[
  F_p\co (W', 0) \to \R n, 0)
\]
that expresses $X'$ by the formula 
 \[
   X'_x=\big(x, F_p(x)\big) \in \{x\}\times \R n, \qquad (x\in W'). 
\]
Noting that $F^{-1} (0)=p$, we define 
$\msf i^{PH}_p(X)\in\ZZ$ as  the degree of the map defined for
any sufficiently small $\eps >0$ as
\[
 \S{n-1}\to\S{n-1},\quad u\mapsto \frac{F_p (\eps u)}{\|F_p (\eps u) \|}\,.
\]
 This degree is independent of $\eps$ and the chart $\phi$, by standard
properties of the degree function.   Therefore the integer 

\[\textstyle
 \msf
i^{PH}_K(X)=i^{PH} (X, U):=
\begin{cases}
        \sum_{p\in K} i^{PH}_p(X) \, &\text{if} \,K\ne\varnothing,\\
        0\, &\text{if} \,K =\empty.
\end{cases}
\]
is well defined and  
depends only on $X$ and $K$. 
\end{definition}

\begin{proposition}             \mylabel{th:xtfru}
Let $\{X^t\}$ be a deformation that is nonsingular in $\fr U$.  If
both $\Z{X^0}\cap  U$ and $\Z{X^1}\cap U$ are finite, then
\[
 i^{PH} (X, U) = i^{PH} (X^1, U).
\]
\end{proposition}
\begin{proof} The proof is similar to that of a standard result on homotopy
  invariance of intersection numbers in oriented manifolds
(compare  {\sc Hirsch} \cite[Theorem   5.2.1]{Hirsch76}). 
\end{proof}

\begin{definition}              \mylabel{th:supp}
The {\em support} of a deformation $\{X^t\}$ is the closed set
\[
 \mathsf {supp}\{X^t\}:=\big\{p\in M\co
X^t_p =X^0_p, \quad{0\le t\le 1} \big\}.
\]
The deformation is {\em compactly supported in  $S$} provided 
$ \mathsf {supp}\{X^t\}$ is a compact subset of $S$. 
\end{definition}

\begin{definition}              \mylabel{th:defindex}
The {\em index of $X$ in $U$} is 
%
\begin{equation}                \label{eq:defindex}
\msf i (X, U):=\msf i^{PH} (X', U)
\end{equation}
%
where $X'$ is any vector field on $M$ such that $\Z {X'}\cap U$ is
finite and there is a deformation from $X$ to $X'$  compactly
supported in $\Int(U)$.  This integer is well defined because the
right hand side of Equation (\ref{eq:defindex}) depends only on $X$
and $U$, by Proposition \ref{th:xtfru}.
The notation $\msf i (X,U)$ tacitly assumes $U$ is isolating for
$X$.  
\end{definition}

\begin{lemma}           \mylabel{th:isol}
If   $U$ and $ U_1$ are  isolating for $(X, K)$ then 
$\msf i (X, U)=\msf i (X, U_1)$.
\end{lemma}
\begin{proof} 
Let $W$ be isolating for $(X, K)$, with $\ov W\subset U_1 \cap U$.
It suffices to show that 
$\msf i (X, U)=\msf i (X, W)$,  for
this also implies $\msf i (X, U_1)=\msf i (X, W$. 
 By definition, 
$\msf i (X, W)=\msf i^{PH}(X', W)$ provided $X$ and $X'$ are
homotopic by deformation with compact support in $W$ and $\Z{Y^1}\cap
W$ is finite.  Let $\{Y^t\}$ be the deformation defined by
\[
Y^t_p=\begin{cases} X^t_p & \text{if $p\notin W$,}\\ Y^t_p & \text{if
  $p\in \ov W$}.
      \end{cases}
\]
Therefore $\msf i (X, U)=\msf i (X, W)$, because this deformation is
compactly supported in $U$ and $\Z{Y^1}\cap U$ is finite.
\end{proof}
It follows that $\msf i (X, U)$ depends only on $X$ and $K$.  The {\em
  index of $X$ at $K$} is
\[
\msf i_K (X):=\msf i (X, U).
\] 
It is easy to see that the index function enjoys the following additivity:

\begin{proposition}             \mylabel{th:add}
Let $K_1, K_2$ be disjoint $X$-blocks, with isolating neighborhoods
$U_1, U_2$ respectively.  Then
\[\begin{split}
\msf i_{K_1\cup K_2} (X) & =  \msf i_{ K_1} (X)   + \msf i_{K_2}  (X), \\
\msf i (X, U_1 \cup U_2) & =  \msf i (X, U_1)   + \msf i (X, U_2). 
\end{split}
\]
\end{proposition}

The following property is crucial: 
\begin{theorem}[{\sc  Stability}]   \mylabel{th:stability}
Let $U\subset M$ be  isolating for $X$. 
\begin{description}

\item[(a)] If $\msf i (X, U)\ne 0$ then $\Z X\cap U\ne\varnothing$. 

\item[(b)] If   $Y$ is sufficiently close to $X$
then $U$ is isolating for $Y$ and   $\msf i (Y,  U)=\msf i (X, U)$.

 \item[(c)]  Let $\{X^t\}$ be a deformation of $X$ that is nonsingular
   in $\fr U$.  Then 
\[\msf i (X^t,  U)=\msf i (X, U), \qquad (0\le t\le 1).
\]
\end{description}
\end{theorem}
\begin{proof} 
If $\msf i (X, U)\ne 0$,  Definition \ref{th:defindex} shows that $X$
is the limit of a convergent sequence $\{X^n\}$ in $\V (M)$ such that
$\Z {X^n}\cap U\ne\varnothing$. Passing to a subsequence and using
compactness of $\ov U$ shows that $\Z X\cap \ov U\ne\varnothing$,
and (a) follows because $\Z X\cap \fr U=\emp$.  Parts (b) and (c)
are implied by Propositions \ref{th:convex} and \ref{th:xtfru}.
\end{proof}

\begin{proposition}             \mylabel{th:lindep}
Assume  $X, Y\in\V (M)$  and $U\subset M$ is
isolating for both $X$ and $Y$. 
For each component $U'$ of $U$ that meets $\Z X \cup \Z Y$, let one
of the following conditions hold:
\begin{description}

\item[(a)] $X_p \ne \lam Y_p, \quad (p\in \fr {U'}, \, \lam <0)$,  
\end{description}
or
\begin{description}
\item[(b)] $X_p \ne \lam Y_p, \quad (p\in \fr {U'}, \, \lam >0)$.

\end{description}
Then $\msf i (X, U)=\msf i (Y, U)$. 
\end{proposition}
\begin{proof}  $U\cap \big(\Z X\cup\Z Y\big)$ is  the compact set
$\ov U \cap \big(\Z X \cup \Z Y \big)$. This implies only finitely
  many components of $U$ meet $\Z X\cup\Z Y$.
The union $U_1$ of these components is isolating for $X$ and $Y$.  The
index function is additive over 
disjoint unions, and both $X$ and $Y$ have index zero in the open set
$U \verb=\=U_1$, which is disjoint from $\Z X\cup\Z Y$. 
Therefore
 \[\begin{split}\msf i (X, U) &=\msf i (X, U_1),\\
                \msf i (Y, U) &=\msf i (Y, U_1).
\end{split}
\]
 Replacing $U$ by $U_1$,  we assume $U$ has only finitely many components.
As it suffices to prove $X$ and $Y$ have the same index in  each
  component of $U$, we also assume $U$ is connected. 

Let 
$p\in \fr U$ be arbitrary.  
If (a) holds, consider the deformation
\[
   X^t:= (1-t)X + tY,\quad (0\le t\le 1).
\]
  If $t= 0$ or $1$ then $X^t_p\ne 0$ because $U$ is isolating for $X$
  and $Y$, while if $0<t<1$ then $X^t_p\ne 0$ by (a).  Therefore the
  conclusion follows from the Stability Theorem \ref{th:stability}.
If (b) holds  the same
argument  works for the deformation $(1-t)X - tY$.
\end{proof}

\begin{proposition}             \mylabel{th:wedge}
Assume $U$ is an isolating neighborhood for both $X$ and $Y$, whose
closure $N$ is a $C^1$ submanifold such that
\begin{equation}                \label{eq:xwry}
X_p\wedge_\RR  Y_p =0, \quad (p\in \p N) 
\end{equation}
and one of the following conditions holds:
\begin{description}

\item[(a)]  $M$ is  even-dimensional

\item[(b)]   $M$ is  odd-dimensional and $X$ and $Y$ are tangent to
  $\p N$. 

\end{description}
Then
$\msf i (X,U)=\msf i(Y, U)$.
\end{proposition}
\begin{proof}
By the Stability Theorem \ref{th:stability} it suffices to find a
deformation from $X$ to $Y$ that is nonsingular on $\p N$.  As $\p N$
is a subcomplex of a smooth triangulation of $M$ ({\sc Whitehead}
\cite{Whitehead40}, {\sc Munkres} \cite{Munkres63}), The Homotopy
Extension Theorem ({\sc Steenrod} \cite [Th. 34.9] {Steenrod51}) shows
that this  deformation exists provided $X|\p N$ and $Y|\p N$ are connected by a
homotopy of nonsingular sections of $T_{\p N}M$.  Such a homotopy
exists  in case (a)
 because the antipodal map and identity maps of $\R n\,\verb=\=\,\{0\}$
 are homotopic, and in case (b) because these maps  in
 $\R{n-1}\,\verb=\=\,\{0\}$ are homotopic. 
\end{proof}

Fix $U$ and $N=\ov U$ as in Proposition \ref{th:wedge}, so that $\msf i
(X, U)=\msf i (X, N\verb=\=\p N)$.  An orientation of $N$ corresponds
to a generator
\[
 \nu_N\in H_n (N,\p N)\cong \ZZ.
\]
Let $\nu^N \in H^n (N,\p N)$ be the dual generator. 

Evaluating cocyles on cycles defines the canonical dual pairing (the
Kronecker Index):
\[
   H^n (N, \p N) \times H_n (N,\p
  N)\to \ZZ, \qquad (c, \lam)\mapsto c\cdot\lam.
\]
Let $c_{X,N}\in H^n(N, \p N)$ be
the obstruction to extending $X|\p N$ to a 
 nonsingular vector field on $N$.   
Unwinding definitions proves:
\begin{proposition}             \mylabel{th:obstruction}
If $N$ is oriented, $\msf i (X, U) = c_{X,N}\cdot \nu_N$. \qed
\end{proposition}
A similar result holds  for nonorientable manifolds, using homology
with coefficients twisted by the orientation sheaf.

\begin{theorem}               \mylabel{th:approx}
If $X$ can be approximated by vector fields $X'$ with no zeros in $\ov
U$ then $\msf i (X, U)=0$, and the converse holds provided $U$ is connected.
\end{theorem}
\begin{proof} If the approximation is possible then the index vanishes
  by  the Stability Theorem \ref{th:stability}.  To prove the
  converse fix a Riemann metric on $M$ and $\eps >0$.  There exists
   an isolating neighborhood  $U'$
  of $K$ whose closure  is a compact submanifold $N\subset U$ and 
\[\|X_p\|<\eps,  \quad (p\in N).
\]
Define
\[E_\eps:=\{x\in\  (N)\co 0< \|x\| <\eps, \quad (p\in N)\}.
\]
This is the total space of a fibre bundle $\eta$ over $N$ that is fibre
homotopically equivalent to the sphere bundle associated to the
tangent bundle of $N$. 

 $X|\p N$ extends to a
section $X''\co N\to E_{\eps}$ of $\eta$, by  Proposition
\ref{th:obstruction}.   Let $X'\in \V (M)$ be the 
extension of $X''$ that agrees with $X$ outside $N$.  Then $X'$ is an
$\eps$-approximation to $X$ with no zeros in $\ov U$.  
\end{proof}

Examination of the proof, together with standard approximation theory,
yields the following addendum to Theorem \ref{th:approx}:
\begin{corollary}               \mylabel{th:approxcor}
Assume $\msf i (X, U)=0$.
\begin{description}

\item[(i)] If $X$ is analytic, the approximations in {\em Theorem
  \ref{th:approx}} can be chosen to be analytic.

\item[(ii)]  If $X$ is $C^r$ and  $0\le r\le \infty$,  the
  approximations can be chosen to be $C^r$ and to agree with $X$ in
  $M\verb=\=U$. \qed

\end{description}
\end{corollary}

\begin{definition}              \mylabel{th:deftriv}
Let $\eta$ denote a real or complex vector bundle with total space
$E$ and $n$-dimensional fibres.  
A {\em trivialization} of $\eta$ is a map $\psi\co E \to \F n$ that
restricts to linear isomorphisms on fibres.
\end{definition}

\begin{proposition}               \mylabel{th:obcor}
Assume 
$N\subset U$ is a compact, connected real $n$-manifold whose
interior is isolating for $(X, K)$. 
Let $\psi$ be a trivialization of
$\tau_{\p N}(M)$.  Then $\msf i (X, U)$ equals the degree $\msf
{deg}(F_X)$ of the map
\[
 F_X\co \p N\to \S {n-1}, \quad p\mapsto \frac{\psi
    (X_p)} {\|\psi (X_p)\|}.
\]
\end{proposition}
\begin{proof} Follows from Proposition \ref{th:obstruction}, because
  $\msf {deg}(F_X)=c_{X,N}\cdot \nu_N$ by obstruction theory. 
\end{proof}

This result will be used in the proofs Theorem \ref{th:MAIN}:
\begin{proposition}         \mylabel{th:fue}
Let $W\subset M$ be a connected isolating neighbrohood for $(X,
K)$.  Assume the following data:
\begin{itemize}
\item $\Phi\co T(W\verb=\=K)\to \R n$ is a trivialization of
  $\tau (W\verb=\= K)$,

\item $E\subset \R {n\times n}$\, is a linear space of
  matrices,\, $\dim (E) < n$,

\item$A\co W\sm K\to E$ is a map such that
\begin{equation}                \label{eq:xpap}
\Phi(X_q) = A(q)\cdot\Phi(Y_q), \qquad (q\in W\sm K).
\end{equation}
\end{itemize}
If $W$ is isolating for $Y\in \V (M)$, then
 \  
$
 \msf i (Y, W)=0\implies \msf i (X, W)=0$.
\end{proposition}
\begin{proof} 
Consider the maps 
\[\begin{split}
 & F_X\co \p N\to \S {n-1}, \quad p\mapsto \frac{\Phi
    (X_p)} {\|\Phi (X_p)\|}, \\
 & F_Y\co \p N\to \S {n-1}, \quad p\mapsto \frac{\Phi
    (Y_p)} {\|\Phi (Y_p)\|}.
\end{split}
\]
 Corollary \ref{th:obcor} implies $\msf {deg} (F_Y)=0$, hence 
$F_Y$ is null homotopic, and it suffices to prove $F_X$ null
 homotopic.  
Degree theory shows that  $F_Y$ is  homotopic to a constant map
\[\tilde F_Y\co\p N\to\S{n-1}, \quad \tilde F_Y(p)= c\in\S{n-1}.
\]
By Equation  (\ref{eq:xpap}) there exists $\lam\co \p
N\to \RR$ such that 
\[
   F_X (p) = \lam (p) A (p) F_Y(p), \quad \lam (p) >0.
\]
Consequently $F_X$ is homotopic to 
\[
 \tilde F_X\co \p N\to \S{n-1}, \qquad \tilde F_X(p)= \lam (p)A
(p)c.
 \]
 The map
\[
   H\co  E\verb=\= \{0\} \to \S{n-1}, \quad B\mapsto\frac{B(c)}{\|B(c)\|} 
\] 
satisfies:
\begin{equation}                \label{eq:fxpn}
\tilde F_X (\p N) \subset H ( E\verb=\=\{0\})\subset\S {n-1}.
\end{equation}
Since the unit sphere $\Sig \subset E\verb=\=\{0\}$ is a deformation
retract of $ E\verb=\= \{0\}$, Equation (\ref{eq:fxpn}) shows that
$\tilde F_X $ is homotopic to a map
\[
G\co \p N \to H (\Sig) \subset\S {n-1}.
\]
Now $\dim (\Sig)= \dim (E)-1 \le n-2$. As $H$ is Lipschitz, $\dim (H
(\Sig)) \le n-2$. Therefore $H (\Sig)$ is a proper subset of
$\Sig^{n-1}$ containing $G(\p N)$, implying $G$ is null homotopic.
The conclusion follows because the homotopic  maps
\[F_X, \tilde F_X, G\co \p N\to\S{n-1}
\]
 have the same degree. 
\end{proof}
\begin{example*}[\bf C]         \mylabel{th:exalg}
Let $\mcal A$ denote a finite dimensional algebra over $\RR$ with multiplication \ $ (a,b)\mapsto a\bullet b$.
Let $X, Y$ be vector fields on a connected open set $U\subset
\mcal A$, whose respective zero sets $K, L$ are compact.
Assume there is a  map $A\co U\to \mcal A$ such that
\[
 X_p=A(p)\bullet Y_p, \qquad (p\in U). 
\]
Then\,  $\msf i_Y (U) = 0 \implies \msf
i_X (U)= 0$, by   Proposition  \ref{th:fue}. 
\end{example*}

\section{Proofs of Theorems  \ref{th:MAIN},  \ref{th:liealg},
 \ref{th:liegroup}}
\mylabel{sec:proofs} 
Henceforth $M$ denotes a connected real or complex 2-manifold with
$\p M=\emp$.

\subsection {Proof of Theorem \ref{th:MAIN}} 
The  hypotheses are:
\begin{itemize}

\item $X$ and $Y$ are analytic vector fields on  $M$,

\item $Y$ tracks $X$,

\item $K$ is an essential $X$-block, 

\end{itemize}
The conclusion is that $\Z Y\cap K\ne\varnothing$.  It suffices to
prove:
\begin{quote}{\em
  $\Z Y$ meets every neighborhood of $K$.}
\end{quote}

 Many sets
$S\subset M$ associated to analytic vector fields, including
 zero sets and dependency sets,  are {\em analytic
  spaces}:  Each point of $S$ has an open
neighborhood $V\subset M$ such that $S\cap V$ is the zero set of an
analytic map $V\to \FF^k$.  This implies $S$ is covered by a locally
finite family of disjoint analytic submanifolds. 

The local topology of analytic
  spaces of  is rather simple, owing to the
 theorem of {\sc {\L}ojasiewicz} \cite {Lo64}:
\begin{theorem} [{\sc Triangulation}]        \mylabel{th:triang}
If  $S$ is a locally finite collection of closed analytic spaces in $M$,
there is a triangulation of $M$ such that each  element of $ S$
a  subcomplex. 
\end{theorem}

 We justify three simplifying assumptions by showing that if any one
 of them is violated the conclusion of Theorem \ref{th:MAIN} holds:

\begin{description}

\item [(A1)] {\em $K$ is connected, and $\chi (K)=0$.} 

$K$ is compact and triangulable and hence has only finitely many
components.  As $K$ is an essential $X$-block, so is some component
(Proposition \ref{th:add}), and we can assume $K$ is that component.
If $\chi (K)\ne 0$, the flow induced by $Y$ on the triangulable space
$K$ fixes a point $p\in \Z Y\cap K$ by Lefschetz's Fixed Point Theorem
({\sc Lefschetz} \cite {Lefschetz37}, {\sc Spanier} \cite{Spanier66},
{\sc Dold} \cite {Dold72}).  This justifies (A1).

Note that (A1) implies $K$ has arbitrarily small connected
neighborhoods $U$ that are isolating for $(X,K)$. 

\item [(A2)] {\em  $\dim_\FF (K)=1$ and $K$ is an analytic submanifold.} 

If $\dim_\FF (K)=0$ then $K$ is a singleton by (A1) and the conclusion
of the theorem is
obvious.  

If $\dim_\FF(K)=2$ then $K=M$ because $X$ is analytic, $M$ is
connected, and both are 2-dimensional.  Therefore $\chi (M) \ne 0$ and
the Poincar\'e-Hopf Theorem  \ref{th:PH}
implies $\Z Y\cap K =\Z Y\ne\varnothing$. 

Let $\dim (K)=1$.  Suppose $K$ is not an analytic submanifold.  Its
singular  set is  nonempty,  finite and $Y$-invariant, and thus contained in $\Z Y\cap K$.

\item [(A3)] {\em $U$ is a connected isolating neighborhood for $(X,
  K)$ and $\Z Y \cap \ov U\subset K$.}

If no such $U$ exists, (A1) implies there is a nested sequence 
  $\{U_j\}$ of connected isolating neighborhoods for
  $(X, K)$ whose intersection is $K$,  and  each $U_j$
  contains a $p_j\in \Z Y\cap\ov {U_j}$.    A subsequence of $\{p_j\}$
  tends to a point of $\Z Y \cap K$ by compactness of $K$.

Note that (A3) implies  $U$ is isolating for $Y$.

\end{description}
Henceforth we assume  (A1), (A2) and (A3).

It suffices to prove
\begin{equation}                \label{eq:iyu0}
\msf i (Y, U)\ne 0,
\end{equation}
because then  (A3) implies $\Z Y\cap U$  meets
$K$. 

Both $K$ and the dependency set $\msf D:=\msf{Dep}_\FF (X, Y)$
(Definition \ref{th:defdepend}) are $Y$-invariant analytic spaces
(Proposition \ref{th:ideal}), and (A2) implies $\dim_\FF (\msf D) = 1$
or $2$.
Because $\dim_\FF(K)=1$ by (A3), one of the following conditions is satisfied:
\begin{description}

\item[(B1)] {\em $K$ is a component of $\msf D$},

\item[(B2)] {\em $\dim_\FF(\msf D)=1$ and $K$ is not a component of $\msf D$},

\item[(B3)] {\em $\dim_\FF(\msf D)=2$}.

\end{description}

{\em Assume} (B1): Choose the isolating neighborhood $U$ so small that
$\fr U \cap \msf D=\emp$.  Proposition \ref{th:lindep} implies $i (Y,
U)=\msf i (X, U) \ne 0$, yielding (\ref{eq:iyu0}).

\smallskip
{\em Assume} (B2):  
Because $\msf D$ and $K$ are 1-dimensional and $K\subset \msf D$, the frontier
in $K$ of $K\cap \big(\ov{\msf D\sm K}\big)$ is   $Y$-invariant and
$0$-dimensional, hence a nonempty subset of $\Z Y\cap K$.
 
\smallskip
{\em Assume} (B3): In this case  $\msf D=M$  because $X$ and $Y$
are analytic and $M$ is connected, hence $X\wedge_\FF Y=0$.

If $M$ is real, Proposition \ref{th:wedge}(a) implies  
\[
  \msf i (Y, U)=\msf i (X, U) \ne 0,
\]
whence $\Z Y\cap U\ne\varnothing$ and $\Z Y\cap K\ne\varnothing$ by
(A3).

This completes the proof of  Theorem \ref{th:MAIN} for real $M$.

Henceforth we assume $M$ is complex.  Therefore 
 (A1) and  (A2) imply
\begin{description}
\item[(C1)] {\em $K$ is a compact connected Riemann surface  of genus $1$,
holomorphically embedded in $U$,}

\item[(C2)] {\em The tangent bundle $\tau (K)$ is a holomorphically trivial
  complex line bundle,}
\end{description}

Note that  (B3) implies $ X_p$ and $Y_p$ are linearly dependent over $\CC$ at
all $p\in M$, because $X$ and $Y$ are analytic and $M$ is connected.
Together with (A3) this implies:

\begin{description}
\item[(C3)] {\em There is an open neighborhood $W\subset U$ of $K$ and a 
  holomorphic map $f\co W\to\CC$ satisfying:}
\[p\in W \implies X_p =f(p)Y_p, \qquad f^{-1} (0)= K.
\]  
\end{description}

Since $K$ is a compact, connected, complex submanifold of $M$ having
codimension $1$, it can be viewed as a divisor of the complex analytic
variety $M$ ({\sc Griffiths \& Harris} \cite{GrifHarris78}).  This
divisor determines a holomorphic line bundle $[K]$ over $M$,
canonically associated to the pair $(M, K)$.
\begin{proposition}             \mylabel{th:div}
{~}
\begin{description}

\item[(i)]  The restriction  of $[K]$  to the
  submanifold  $K$ is holomorphically isomorphic to the
algebraic normal bundle $\nu (K, M):=\tau_K(M)/\tau (K)$ of $K$.

\item[(ii)]  $[K]$  is  holomorphically trivial. 

\item[(iii)]   $\nu(K, M)$ is  holomorphically trivial. 

\end{description}
\end{proposition}
\begin{proof}
 Working through the definition of $[K]$ in 
 \cite{GrifHarris78} demonstrates (i).  Part (ii) follows from (C3)
 and the italicized statement on \cite[page  134]{GrifHarris78}), and (ii)
 implies  (iii).\footnote{
An elegant explanation was kindly supplied by  
{\sc D. Eisenbud} \cite{Eisenbud15}:  
 The ideal defining $[K]$ is the dual $\mcal K^*$ of the sheaf
$\mcal K$ of ideals of the analytic space  $K$.
 Because $\mcal K$  is generated by the single
function $f$ it is a product sheaf, and so also is $\mcal K^*$.  Therefore  
 $[K]$ is a  holomorphically trivial line bundle, as is its restriction to
$K$, which is $\nu(K, M)$.} 
\end{proof}

From  (C1), (C2) and Proposition \ref{th:div}(ii) with
$V:=K$ we see that  the
complex vector bundle $\tau_K(M)\cong \tau (K) \oplus \nu (K, M)$ is
holomorphically trivial. 
  As $K$ is triangulable we can choose $W$ in (C3) so that it admits
  $K$ as a deformation retract.  Therefore:
\begin{description}
\item[(C4)]
{\em $\tau (W)$ is a trivial complex vector bundle}
\end{description}
by  the Homotopy Extension Theorem ({\sc Steenrod}
  \cite[Thm. 34.9] {Steenrod51}, {\sc Hirsch} \cite [Chap. 4,
    Thm. 1.5] {Hirsch76}).

Define  
\[
 \theta\co\CC\to\R {2\times 2}, \quad  a+b\sqrt{-1} \mapsto
 \, \left[\begin{smallmatrix} 
   a & -b \\ b & \, a
   \end{smallmatrix}\right], \qquad (a, b\in \RR). 
\]
Let
\[\Theta\co \C {2\times 2} \to \R {4 \times 4}
\]
be the $\RR$-linear isomorphism that replaces each matrix entry \,$z$\,
 by the $2\times 2$ block $\theta(z)$.

Define
\[
 H\co \CC \to \R {4\times 4},  \quad a+b\sqrt{-1} \mapsto
\left[\begin{smallmatrix} 
   a & -b & 0 & 0 \\
   b  & \, a & 0 & 0 \\
   0 & 0  & a & -b\\
   0 & 0  & b & \, a 
  \end{smallmatrix}\right], \quad a, b \in \RR.
\]
Note that $E:=H (\CC)$ is a 2-dimensional linear subspace of $\R
{4\times 4}$. 

Let  
\[\Psi\co T (W)\to\C 2 
\]
be a trivialization of the complex vector bundle $\tau (W)$
(Definition \ref{th:deftriv}).  The real vector bundle $\tau (W^\RR)$
has the trivialization
\[
  \Phi:= \Theta\circ \Psi\co T (W)\to \R 4. 
\]

Let $f\co W \to \CC$ be as in (C3) and set
$ A:= H\circ f\co W \to E$. 
Then
\[
 \Phi (X_q)=A(q)\cdot\Phi (Y_q), \qquad (q\in W).
\]
This implies $\msf i (Y, W) =\msf i (X, W)\ne 0$  by  Proposition \ref{th:fue}, 
because $\msf i (X,W)\ne 0$,  Therefore (\ref{eq:iyu0}) holds,
completing the proof of Theorem \ref{th:MAIN}.  

\subsection{Proof of Theorem \ref{th:liealg} }   \mylabel{sec:liealg}
Recall the hypotheses: 
\begin{itemize}
\item $M$ is a connected real or complex 2-manifold with empty boundary
\item
$\mcal G\subset \V^\om (M)$ is a Lie algebra over ground field $\FF$
  that tracks $X\in \V^\om   (M)$
\item If $M$ is real, $\mcal G$ is supersolvable. 
\item $K$ is an essential $X$-block 
\end{itemize}
To be proved:  $\Z{\mcal G}\cap K\ne\varnothing$.

Consider first the case that $M$ is complex.  We can assume: 

\begin{description}

\item[(A1$'$)] {\em $K$ is connected and $\chi (K)=0$.}

For otherwise the argument used above to justify (A1) shows that there
is point of $K$ fixed by the local flow of every $Y\in \V^1 (M)$
that tracks $X$. 

\end{description}

\begin{description}

\item [(A2$'$)] {\em  $\dim_\FF (K)=1$, 
and  $K$ is an analytic submanifold.}

If $\dim_\FF (K)=0$ then $K$ is a singleton by (A1$'$) and the conclusion
of the theorem is 
obvious.  
If $\dim_\FF(K)=2$ then $K=M$ because $X$ is analytic and $M$ is
connected.  But then $X=0$, contradicting $\msf i_K
(X)\ne\varnothing$. 
 Thus we can assume $\dim
(K)=1$.  If $K$ is not an analytic submanifold, its singular set is a
nonempty and $\mcal G$-invariant; being $0$-dimensional, it is finite hence
contained in $\Z {\mcal G}$.
\end{description}

(A1$'$) and [(A2$'$) imply:

\begin{description}

\item[(C1$'$)] {\em $K$ is a compact connected Riemann surface  of genus $1$,
holomorphically embedded in $M$.} 
\end{description}

As every $Y\in \mcal G$ tracks $X$, $K$ is $Y$-invariant (Proposition
\ref{th:ideal}. Therefore $Y|K$ is a holomorphic
vector field on $K$, and $\Z Y \cap K\ne\varnothing$  (Theorem
\ref{th:MAIN}).  Since a nontrivial holomorphic vector field on a
compact Riemann surface has no zeros,  $K\subset \Z Y$
for all $Y\in \mcal G$. This completes the proof
of Theorem \ref{th:liealg} for complex manifolds.

Now assume: {\em $M$ is real and $\mcal G$ is supersolvable.}

Let $\dim\, G=d, \, 1\le d <\infty$.  Finite dimensionality and the
assumption that $\mcal G$ tracks $X$ implies $X$ spans an ideal $\mcal
H_1$

The conclusion of the theorem is trivial if $d=1$.  If $d=2$ then
$\mcal G$ has a basis $\{X, Y\}$ and Theorem \ref{th:MAIN} implies $\Z
Y\cap K\ne\varnothing$, whence $\Z{\mcal G}\cap K\ne\varnothing$.
 
Proceeding inductively we assume $d >2$ and that the conclusion holds
for smaller values of $d$.  Supersolvability implies there is a chain
of ideals
\[
 \mcal H_1\subset \dots \subset H_d=\mcal G
\]
such that $\dim,\mcal H_k = k$.  Note that
\[
 \Z{\mcal H_1} = K.
\]
The inductive hypothesis implies
\[ L:=\Z {\mcal H_{d-1}}\cap K\ne\varnothing,
\]
and $L$ is $\mcal G$-invariant because the zero sets of the ideals
$\mcal H_{d-1}$ and $\mcal H_1$ are $\mcal G$-invariant.

If $\dim\, L=0$ it is a nonempty finite $\mcal G$ invariant set, hence
contained in $\mcal G \cap K$.  Suppose $\dim\,L=1$.  If $L=K$ there
is nothing more to prove, and if $L\ne K$ its frontier in $K$ is a
nonempty finite $\mcal G$-invariant set.  The proof of Theorem
\ref{th:MAIN} is complete. \qed

\subsection{Proof  of Theorem \ref{th:liegroup}}
The effective analytic action of $G$ on $M$ induces an isomorphism
from the Lie algebra of $G$ onto a Lie algebra $\mcal G\subset \V^\om
(M)$.  Let $X\in \mcal G$ span the Lie algebra of a $1$-dimensional
ideal.  Because $\chi (M)\ne 0$, the set $K:=\Z X$ is an essential
$X$-block by the Poincar\'e-Hopf Theorem \ref{th:PH}.  Theorem
\ref{th:liealg} shows that $\Z{\mcal G}\cap K\ne\varnothing$.
Connectedness of $G$ implies $\Z{\mcal G}=\Fix \alpha$, implying the
conclusion. \qed



\begin{thebibliography}{99}

\bibitem{Belliart97} M.\ Belliart, {\em Actions sans points fixes sur les
surfaces compactes}, Math.\ Z.\ {\bf 225} (1997), 453--465


\bibitem{Bonatti92} C.\ Bonatti, {\em Champs de vecteurs analytiques
commutants, en dimension $3$ ou $4$: existence de z\'{e}ros communs},
Bol.\ Soc.\ Brasil.\ Mat. (N.\ S.) {\bf 22} (1992), 215--247

 \bibitem{Borel56} A.\ Borel, {\em Groupes lin\'eaires algebriques,}
    Ann.\  Math.\ {\bf 64} (1956), 20--80 

\bibitem{Dold72} A. Dold, ``Lectures on Algebraic Topology.'' Die
  Grundlehren der matematischen Wissenschaften Bd. 52, second edition.
  Springer, New York 1972

\bibitem{Eisenbud15} D.\ Eisenbud, personal communication (2015)

\bibitem{Gottlieb86} D.\ Gottlieb, {\em A de Moivre like formula for
  fixed point theory}, in: ``Fixed Point Theory and its Applications
  (Berkeley, CA, 1986).''   Contemporary Mathematics {\bf 72}
  Amer.\ Math.\ Soc., Providence, RI  1988
  
 \bibitem{GrifHarris78} P.\ Griffiths \& J.\ Harris,  ``Principles of
   Algebraic Geometry.'' John Wiley \& Sons, New York 1978


\bibitem{Hirsch76} M.\ Hirsch, ``Differential Topology.'' Springer-Verlag, New
 York 1976

\bibitem{HW2000} M.\ Hirsch \& A.\,Weinstein, {\em Fixed points of
analytic actions of supersoluble Lie groups on compact surfaces},
Ergod.\ Th.\ Dyn.\ Sys.\ {\bf 21} (2001), 1783--1787

\bibitem{Hirsch2010}  M.\ Hirsch,   {\em Actions of Lie groups
  and Lie algebras on manifolds},  in
``A Celebration of the Mathematical Legacy of Raoul Bott.'' 
Centre de Recherches Math\'{e}matiques, U. de Montr\'{e}al.
Proceedings \& Lecture Notes  {\bf 50}, (P. R. Kotiuga, ed.),
Amer. Math. Soc. Providence RI \,2010

\bibitem{Hirsch2014}  M.\ Hirsch, {\em
Fixed points of local actions of nilpotent Lie
groups on surfaces}, arxiv.org/abs/1405.2331 (2013)


\bibitem{Hopf25} H.\ Hopf, {\em Vektorfelder in Mannifgfaltigkeiten},
Math.\ Annalen {\bf 95} (1925), 340--367

\bibitem{Jubin09} B.\ Jubin, {\em A generalized Poincar\'e-Hopf index
  theorem,}   
arxiv.org/0903.0697 (2009)


\bibitem{Lefschetz37} S. Lefschetz, {\em  On the fixed point formula,}
  Ann. Math. {\bf  38} (1937) 819--822


\bibitem{Lima64} E.\ Lima, {\em Common singularities of commuting vector
fields on $2$-manifolds}, Comment.\ Math.\ Helv.\ {\bf 39 } (1964),
97--110


\bibitem{Lo64} S.\ {\L}ojasiewicz, {\em Triangulation of semi-analytic
sets},  Ann.\ Scuola Norm. Sup. Pisa (3)  {\bf 18} (1964), 449--474
 
\bibitem{Molino93}  P.\ Molino, Review of Bonatti \cite {Bonatti92},
  Math Reviews 93h:57044, Amer.\ Math.\ Soc.\ (1993)

\bibitem{Morse29} M.\ Morse, {\em Singular Points of Vector Fields
  Under General Boundary Conditions,}  Amer.\ J.\ Math.\ {\bf 52} (1929),
  165--178

\bibitem{Munkres63} J.\  Munkres, ``Elementary Differential
  Topology.'' Annals Study 54, Princeton University Press, Princeton
  New Jersey 1963

\bibitem{Plante86} J.\ Plante, {\em Fixed points of Lie group actions on
surfaces,} Erg.\ Th.\ Dyn.\ Sys.\ {\bf 6} (1986), 149--161

\bibitem{Poincare85} H.\ Poincar\'e, {\em Sur les courbes d\'efinies par une
\'equation diff\'erentielle,} J.\ Math.\ Pures Appl.\  {\bf 1} (1885),
167--244

\bibitem{Pugh68} C.\ Pugh, {\em A generalized Poincar\'{e} index
  formula,} Topology {\bf 7} (1968), 217--226 


\bibitem{Sommese73} A.\ Sommese, {\em Borel's fixed point theorem for
  Kaehler manifolds and an application,} 
  Proc.\ Amer.\ Math.\ Soc.\ {\bf 41} (1973), 51--54


\bibitem{Spanier66} E. Spanier, ``Algebraic Topology,'' McGraw-Hill, New
York (1966).  

\bibitem{Steenrod51} N. Steenrod, ``The Topology of Fiber Bundles.''
Princeton University Press, Princeton NJ 1951
\bibitem{Turiel03} F.-J.\ Turiel, {\em Analytic actions on compact
  surfaces and fixed points}, Manuscripta Mathematica {\bf 110} (2003),
  195--201

\bibitem{Whitehead40} J.\ H.\ C.\  Whitehead, {\em On $C^1$
  complexes}, Annals of   Math.\ {\bf 41}  (1940), 809--824

\bibitem{Wilson69} F.\ W.\ Wilson, {\em Smoothing derivatives of functions and
  applications,} Trans. Amer. Math. Soc. {\bf 139} (1969), 413--428

\end{thebibliography}
\end{document}